\documentclass{amsart}
\usepackage{amssymb,amsmath}
\usepackage{graphicx}
\usepackage{tikz}

\newcommand{\RR}{\mathbb{R}}

\newcommand{\XX}{\mathcal{X}}
\newcommand{\PSD}{\mathcal{S}_+}
\newcommand{\rank}{\textup{rank}\,}
\newcommand{\rankplus}{\textup{rank}_+\,}
\newcommand{\rankpsd}{\textup{rank}_{\textup{psd}}\,}
\newcommand{\diag}{\textup{diag}}

\newcommand{\sqrtrank}{\textup{rank}_{\! \! {\sqrt{\ }}}\,}
\newcommand{\STAB}{\textup{STAB}}
\newcommand{\xx}{\mathbf x}

\newtheorem{theorem}{Theorem}[section]

\newtheorem{lemma}[theorem]{Lemma}
\newtheorem{corollary}[theorem]{Corollary}
\newtheorem{proposition}[theorem]{Proposition}
\theoremstyle{definition}
\newtheorem{definition}[theorem]{Definition}
\newtheorem{example}[theorem]{Example}

\theoremstyle{remark}
\newtheorem{remark}[theorem]{Remark}

\title{Polytopes of Minimum Positive Semidefinite Rank}

\author{Jo{\~a}o Gouveia}
\address{CMUC, Department of Mathematics,
  University of Coimbra, 3001-454 Coimbra, Portugal}
\email{jgouveia@mat.uc.pt} 

\author{Richard Z. Robinson}
\address{Department of Mathematics, University of Washington, Box
  354350, Seattle, WA 98195, USA} \email{rzr@uw.edu}

\author{Rekha R. Thomas}
\address{Department of Mathematics, University of Washington, Box
  354350, Seattle, WA 98195, USA} \email{rrthomas@uw.edu}

\thanks{Gouveia was partially supported by the Centre for Mathematics at the University of Coimbra and Fundac\~ao para a Ci\^encia e a
Tecnologia, through the European program COMPETE/FEDER, Robinson by 
the U.S. National Science Foundation Graduate Research Fellowship under Grant No. DGE-0718124, and Thomas by 
the U.S. National Science Foundation grant DMS-1115293.}

\begin{document}

\begin{abstract}
The positive semidefinite (psd) rank of a polytope is the smallest $k$ for which the cone of $k \times k$ real symmetric psd matrices 
admits an affine slice that projects onto the polytope. In this paper we show that the psd rank of a polytope is at least the dimension of the polytope plus one, and we characterize those polytopes whose psd rank equals this lower bound. We give several classes of polytopes that achieve the minimum possible psd rank including a complete characterization in dimensions two and three.
\end{abstract}

\maketitle

\section{Introduction}

Efficient representations of polytopes are of fundamental importance in contexts such as linear optimization
where the complexity of many algorithms depends on the size of the representation. A standard idea to find a compact description of a complicated polytope $P \subset \RR^n$ is to look for a simpler convex set of 
higher dimension  that has $P$ as a linear image of it.
Affine slices of closed convex cones offer a rich source of convex sets and the following definition was introduced in \cite{GPT2}. 

\begin{definition} \label{def:K-lift}
Let $P \subset \RR^n$ be a polytope. If $K \subset \RR^m$ is a closed convex cone, $L$ an affine space in $\RR^m$, and $\pi \,:\, \RR^m \rightarrow \RR^n$ a linear map such that $P = \pi(K \cap L)$, then we say that $K \cap L$ is a $K$-{\em lift} of $P$.
\end{definition}

If linear optimization over affine slices of $K$ admits efficient algorithms, then often, linear optimization over $P$ can be done rapidly as well. Well studied cones in this context are 
nonnegative orthants and the cones of real symmetric positive semidefinite (psd) matrices. We will denote the $m$-dimensional nonnegative orthant by $\RR^m_+$ and the cone of $m \times m$ psd matrices by $\PSD^m$. 
Affine slices of $\RR^m_+$ are polyhedra over which linear optimization can be done efficiently via {\em linear programming}. Affine slices of $\PSD^m$ are called {\em spectrahedra}, and linear optimization over them can be done efficiently via {\em semidefinite programming}. Recall that $\RR^m_+$ embeds into $\PSD^m$ via diagonal matrices and hence, polyhedra are special cases of spectrahedra, and semidefinite programming generalizes linear programming. 

There are many families of polytopes in $\RR^n$ with exponentially many facets (in $n$) that admit small (polynomial in $n$) polyhedral or spectrahedral lifts. 
Examples are the {\em parity} and {\em spanning tree polytopes} \cite{Yannakakis}, the {\em permutahedron} \cite{Goemans} and the {\em stable set polytope} of a {\em perfect graph} \cite{LovaszSchrijver91}. When the lifts come from families of cones such as $\{\RR^m_+\}$ or $\{ \PSD^m \}$, it is useful to determine the smallest cone in the family that admits a lift of the polytope. This allows the notion of {\em cone rank} of a polytope with respect to a family of cones \cite{GPT2}. We recall the definitions needed in this paper.

\begin{definition} \label{def:ranks} \cite{GPT2}
\begin{enumerate}
\item The {\em nonnegative rank} of a polytope $P \subset \RR^n$, denoted as $\rankplus P$, is the smallest $k$ such that $P$ has an $\RR^k_+$-lift.
\item The {\em positive semidefinite rank} of a polytope $P \subset  \RR^n$, denoted as $\rankpsd P$,  is the smallest $k$ such that $P$ has an $\PSD^k$-lift.
\end{enumerate}
\end{definition}

To describe our results, we need the following further definitions.

\begin{definition} \label{def:slack matrix} \cite{Yannakakis}
Let $P$ be a full-dimensional polytope in $\RR^n$ with vertex set 
$\{p_1,\ldots,p_v\}$ and an irredundant (facet) inequality representation
\[
P=\left\{ x \in \RR^n :  \beta_1 - \langle a_1, x \rangle  \geq 0 , \ldots, \beta_f - \langle a_f, x \rangle  \geq 0 \right\}
\]
where $\beta_j \in \RR$ and $a_j \in \RR^n$.
Then the nonnegative matrix  in $\RR^{v \times f}$ whose $(i,j)$-entry
is $\beta_j - \langle a_j, p_i \rangle$ is called a {\em slack matrix of $P$}. 
\end{definition}

Recall that the {\em polar dual} of a cone $K \subset \RR^m$ is the cone $$K^* := \{ y \in \RR^m \,:\, \langle x, y \rangle \geq 0 \,\,\,\forall \,\,\, x \in K \}.$$ 
In the vector space of $m \times m$ symmetric matrices we use the trace inner product $\langle A, B \rangle = \textup{Tr}(AB)$. Both $\PSD^k$ and $\RR^k_+$ are {\em self dual} cones, meaning that $K^*=K$, and we will identify them with their polar duals in what follows.  The notion of {\em cone factorizations} of slack matrices plays a central role in the theory of cone lifts of polytopes.

\begin{definition} \label{def:K-factorization} \cite{GPT2}
Let $M = (M_{ij}) \in \RR_+^{p \times q}$ be a nonnegative matrix and $K$ a
closed convex cone whose polar dual is $K^*$.
\begin{itemize}
\item  A $K$-{\em factorization} of $M$ is a pair of
ordered sets $a^1, \ldots, a^p \in K$ and $b^1, \ldots, b^q \in K^*$ (called {\em factors})
such that $\langle a^i, b^j \rangle = M_{ij}$. 
\item When $K = \RR^m_+$ (respectively, $\PSD^m$), a $K$-factorization of $M$ is called a  
{\em nonnegative} (respectively, {\em psd}) {\em factorization} of $M$.
\item The smallest $k$ for which $M$ has an $\RR^k_+$-factorization (respectively, $\PSD^k$-factorization) 
is called the {\em nonnegative rank} (respectively, {\em psd rank}) of $M$.  We denote these invariants of $M$ as $\rankplus M$ and 
$\rankpsd M$. 
\end{itemize}
\end{definition}

Any positive scaling of a facet inequality of a polytope $P$ can be used in Definition~\ref{def:slack matrix} and so the slack matrix of $P$ is only defined up to positive scalings of its columns. 
We denote any such slack matrix of $P$ by $S_P$. Since scaling rows or columns of a matrix $M$ by arbitrary positive real numbers does not affect the existence of a $K$-factorization of $M$, all slack matrices of $P$ will have the same behavior with respect to $K$-factorizations and, in particular, have the same nonnegative (respectively, psd) rank. 

In what follows, $P \subset \RR^n$ is always an $n$-dimensional polytope. 
Yannakakis showed in \cite{Yannakakis} that $\rankplus P = \rankplus S_P$ by proving that $P$ has an $\RR^k_+$-lift if and only if $S_P$ has an $\RR^k_+$-factorization. 
The nonnegative rank of a polytope has been the subject of many recent papers \cite{FKPT, FMPTW, FioriniRothvossTiwary, GillisGlineur, KaibelPashkovich}. The psd rank of a {\em convex set} $C \subset \RR^n$ was introduced in \cite{GPT2} where Yannakakis' theorem was generalized (Theorem 2.4 \cite{GPT2}). Specializing to polytopes, this theorem says that $P$ has a $K$-lift (in particular, $\PSD^k$-lift)
 if and only if $S_P$ has a $K$-factorization ($\PSD^k$-factorization), and so, $\rankpsd P = \rankpsd S_P$.  (The extension of Yannakakis' theorem in the case of polytopes also appeared in \cite{FMPTW}.) Since $\RR^k_+$ embeds into $\PSD^k$ for each $k$, we always have $\rankpsd P \leq \rankplus P$. It is easy to see that $\rankplus P \geq \rank S_P = n+1$. In Proposition~\ref{prop:lower bound on psd rank} we 
 show that $\rankpsd P$ is also at least $n+1$. This is not immediate since for a general nonnegative matrix $M$, $\rank M$ is not a lower bound for $\rankpsd M$, and the correct relationship is that $\frac{1}{2}(\sqrt{1 + 8 \rank M}-1) \leq \rankpsd M$ \cite{GPT2}.
Theorem~\ref{thm:psdrank n+1} characterizes those $n$-polytopes whose psd rank equals $n+1$, and we give several families of $n$-dimensional polytopes whose psd rank equals this lower bound.

We now recall a few useful facts about nonnegative and psd ranks of polytopes that will be needed in this paper. It follows from \cite[Prop.~2]{GPT2} that $\rankplus P$ and $\rankpsd P$ are invariant under projective (and hence also, affine) transformations of $P$. Further, transposing a matrix $M$ does not effect the existence of a $K$-factorization of $M$ if $K$ is self-dual. Therefore, if $P$ contains the origin in its interior, its {\em polar} polytope is 
$P^\circ := \{ y \in \RR^n \,:\, \langle x, y \rangle \leq 1 \,\,\forall \,\, x \in P \}$, and $\rankplus P = \rankplus P^\circ$ and $\rankpsd P = \rankpsd P^\circ$ since we can obtain a slack matrix of $P^\circ$ by transposing a slack matrix of $P$ and rescaling rows. It is common to define the slack matrix of a polytope using any inequality description of the polytope, including redundant inequalities. This will not affect the nonnegative or psd rank of the polytope. However, since some of our results will become more cumbersome to state using this more general definition of a slack matrix, we restrict ourselves to Definition~\ref{def:slack matrix}.

The psd rank of a polytope $P$ quantifies the power of semidefinite programming to provide efficient algorithms for linear optimization over $P$. For example, the stable set polytope of a perfect graph on $n$ vertices is known to have psd rank $n+1$ which provides the only known polynomial time algorithm (via semidefinite programming) for finding the highest weight stable set in a perfect graph.  
The connection between psd rank and semidefinite lifts allows psd rank to become a possible tool for settling questions concerning semidefinite programming in combinatorial optimization. A question that is currently active is whether the nonnegative rank of the {\em perfect matching polytope} of a complete graph $K_n$ is polynomial in $n$. This was raised in \cite{Yannakakis} where it was shown that there are no small symmetric $\RR^k_+$-lifts of these polytopes. Both nonnegative and psd ranks of these polytopes are unknown at the moment. Another active question concerns the possible gap between $\rankplus P$ and $\rankpsd P$ which is a measure of the relative strength of linear vs. semidefinite programming for linear optimization over $P$. No example where this gap is large is known so far. While nonnegative rank has been studied in several papers, the notion of psd rank is new.
The results and techniques presented here further our understanding of psd rank of a polytope.

This paper is organized as follows. In Section~\ref{sec:general matrices} we introduce tools to study the psd rank of a general nonnegative matrix $M$ using Hadamard square roots of $M$. In Section~\ref{sec:slack matrices}, we specialize to slack matrices of polytopes and derive the lower bound of $n+1$ for the psd rank of a $n$-dimensional polytope (Proposition~\ref{prop:lower bound on psd rank}). Theorem~\ref{thm:psdrank n+1} characterizes $n$-dimensional polytopes with psd rank $n+1$ in terms of the lowest rank of a Hadamard square root of a slack matrix of $P$. In Section~\ref{sec:examples} we give several families of polytopes whose psd rank equals this lower bound. In the plane, the full-dimensional polytopes with psd rank three are exactly triangles and quadrilaterals (Theorem~\ref{thm:polygons of psd rank 3}). Every polytope in $\RR^n$ with at most $n+2$ vertices has psd rank $n+1$ (Theorem~\ref{thm:n+2 vertices}). In $\RR^3$, the situation gets more tricky and we exhibit polytopes of a fixed combinatorial type (octahedra) whose psd rank depends on the embedding of the polytope. Nonetheless, we show that the three dimensional polytopes with psd rank four are exactly tetrahedra, quadrilateral pyramids, bisimplicies, combinatorial triangular prisms, ``biplanar'' octahedra, and ``biplanar'' cuboids (Theorem~\ref{thm:3d polytopes of psd rank four}). It follows from \cite{GPT1} that if $S_P$ is a $0/1$ matrix then $\rankpsd P = n+1$. Such polytopes are called $2$-level polytopes and include the stable set polytopes of perfect graphs. We exhibit polytopes that are not combinatorially equivalent to $2$-level polytopes whose psd rank achieves the lower bound. We also show polytopes that are combinatorially equivalent to $2$-level polytopes whose psd rank is not the minimum possible.  Finally, we prove in Theorem~\ref{thm:perfect graphs} that for stable set polytopes, the results of Lov{\'a}sz prevail even in our general setting in the sense that the stable set polytope of a graph on $n$ vertices has psd rank $n+1$ if and only if the graph is perfect.

\section{Hadamard square roots and psd ranks of matrices}
\label{sec:general matrices}

\begin{definition} \label{def:Hadamard square root} 
A {\em Hadamard square root} of a nonnegative real matrix $M$, denoted as $\sqrt{M}$,  is any matrix 
whose $(i,j)$-entry is a square root (positive or negative) of the $(i,j)$-entry of $M$.  Additionally, we 
let $\sqrt[+]{M}$ denote the all-nonnegative Hadamard square root of $M$.
\end{definition}

Let $\sqrtrank M := \textup{min} \{ \rank \sqrt{M} \}$ be the minimum rank of a Hadamard square root of a nonnegative matrix $M$. We recall the basic connection between the psd rank of a nonnegative matrix $M$ and $\sqrtrank M$ shown in \cite[Proposition 4.8]{GPT2} , and also in \cite{FMPTW}.

\begin{proposition} \label{prop:psd rank and Hadamard square roots}
If $M$ is a nonnegative matrix, then $\rankpsd M \leq \sqrtrank M $.
In particular, the psd rank of a $0/1$ matrix is at most the rank of the matrix.
\end{proposition}

\begin{proof} Let $\sqrt{M}$ be a Hadamard square root of $M \in \RR^{p \times q}_+$ of rank $r$. Then there exist vectors 
$a_1, \ldots, a_p, b_1, \ldots, b_q \in \RR^{r}$ such that $(\sqrt{M})_{ij} = \langle a_i, b_j \rangle$. Therefore,  
$M_{ij} = \langle a_i, b_j \rangle^2 = \langle a_i a_i^T, b_j b_j^T \rangle$ where the second inner product is the trace inner product for symmetric matrices defined earlier.  Hence, $\rankpsd{M} \leq r$.
\end{proof}

The upper bound in Proposition~\ref{prop:psd rank and Hadamard square roots} can be strict even for simple examples.

\begin{example} \label{ex:rank 3 and psd rank 2} For the matrix
\[ M := \left[ \begin{array}{ccc}
1 & 1 & 1 \\
1 & 0 & 1 \\
0 & 1 & 1 \end{array} \right], \] 
$\rank M  = \sqrtrank M = 3$ while $\rankpsd M  = 2$.
Assigning the first three psd matrices below to the rows of $M$, and the next three to the columns of $M$, we obtain a 
$\PSD^2$-factorization of $M$:
$$ \left[ \begin{array}{cc} 0.5 & -0.5 \\ -0.5 & 1 \end{array} \right], \left[ \begin{array}{cc} 0.5 & 0 \\ 0 & 0 \end{array} \right], 
\left[ \begin{array}{cc} 0 & 0 \\ 0 & 1 \end{array} \right] \textup{ and } 
\left[ \begin{array}{cc} 2 & 0 \\ 0 & 0 \end{array} \right], \left[ \begin{array}{cc} 0 & 0 \\ 0 & 1 \end{array} \right], \left[ \begin{array}{cc} 2 & 1 \\ 1 & 1 \end{array} \right].$$ 
\end{example}

\bigskip

Even though $\sqrtrank M$ is only an upper bound on $\rankpsd{M}$, we cannot find $\PSD^k$-factorizations of $M$ with only rank one factors if $k 
< \sqrtrank M$ as shown in Lemma~\ref{lem:rank one factors} below. Note that the psd factors corresponding to the first row and the third column of  the matrix $M$ in Example~\ref{ex:rank 3 and psd rank 2} both have rank two. 

\begin{lemma} \label{lem:rank one factors}
The smallest $k$ for which a nonnegative real matrix $M$ admits a $\PSD^k$-factorization in which all factors are matrices of rank one is $k = \sqrtrank M$.
\end{lemma}

\begin{proof} If $k = \sqrtrank M$, then there is a Hadamard square root of $M \in \RR^{p \times q}_+$ of rank $k$ and the proof of Proposition~\ref{prop:psd rank and Hadamard square roots} gives a $\PSD^k$-factorization of $M$ in which all factors have rank one.
On the other hand, if there exist $a_1a_1^T, \ldots, a_pa_p^T, b_1b_1^T, \ldots, b_qb_q^T \in \PSD^k$ such that $M_{ij}= \langle a_ia_i^T, b_jb_j^T \rangle = \langle a_i, b_j \rangle^2$, then the matrix with $(i,j)$-entry $\langle a_i, b_j \rangle$ is a Hadamard square root of $M$ of rank at most $k$. 
\end{proof} 

\begin{example} \label{ex:derangement matrix}
For a $0/1$ matrix $M$, $\rankpsd M \leq \sqrtrank M \leq \rank M$. In Example~\ref{ex:rank 3 and psd rank 2} we saw that the first inequality may be strict. We now show that the second inequality may also be strict.  
The following {\em derangement} matrix 
\[ \left[ \begin{array}{ccc}
0 & 1 & 1 \\
1 & 0 & 1 \\
1 & 1 & 0 \end{array} \right] \]
has rank three and psd rank two. An $\PSD^2$-factorization in which all factors have rank one is gotten by assigning 
$$ \left[ \begin{array}{cc} 0 & 0 \\ 0 & 1 \end{array} \right], \left[ \begin{array}{cc} 1 & 0 \\ 0 & 0 \end{array} \right], \left[ \begin{array}{cc} 1 & 1 \\ 1 & 1 \end{array} \right], \left[ \begin{array}{cc} 1 & 0 \\ 0 & 0 \end{array} \right],\left[ \begin{array}{cc} 0 & 0 \\ 0 & 1 \end{array} \right], \left[ \begin{array}{cc} 1 & -1 \\ -1 & 1 \end{array} \right] $$ to the three rows and the three columns, respectively. A Hadamard square root of $M$ of rank two is 
$$ \left[ \begin{array}{rrr} 
0 & -1 & 1\\
1 & 0 & 1\\
1 & 1 & 0
\end{array} \right].
$$ 
\end{example}

We now show a method to increase the psd rank of any matrix by one. This technique will be used later to study the psd rank of a polytope.

\begin{proposition} \label{prop:extending rank}
Suppose $M \in \RR^{p \times q}_+$ and $\rankpsd M = k$. If $M$ is extended to $M' = \left( \begin{array}{cc} M & {\bf 0} \\ w & \alpha \end{array} \right)$ 
where $w \in \RR_+^q$, $\alpha > 0$ and ${\bf 0}$ is a column of zeros, then $\rankpsd M' = k+1$. Further, the factor 
associated to the last column of $M'$ in any $\PSD^{k+1}$-factorization of $M'$ has rank one.
\end{proposition}

\begin{proof}
Suppose $M'$ has a $\PSD^k$-factorization with factors $A_1, \ldots, A_p, A \in \PSD^k$ associated to its rows and $B_1, \ldots, B_q,B \in \PSD^k$ associated to its columns. Then $A, B \neq 0$ since $\langle A, B \rangle = \alpha \neq 0$. Let $r= \rank(B) >0$. Then 
there exists an orthogonal matrix $U$ such that $U^{-1} B U = \diag(\lambda_1, \ldots, \lambda_{r}, 0, \ldots, 0) =: D$ where $\lambda_1, \ldots, \lambda_r$ are the nonzero (positive) eigenvalues of $B$.
Let $A_i' := U^{-1}A_iU$ for $i=1,\ldots,p$. Then 
$$\langle D,A_i' \rangle = \textup{Tr}(U^{-1}BA_iU) = \textup{Tr}(BA_i) = \langle B, A_i \rangle = 0 \,\,\,\forall \,\,\,i=1,\ldots,p.$$ 
Since the diagonal entries of $A_i'$ are nonnegative, $\langle D, A_i' \rangle = 0$ implies that the first $r$ diagonal entries of $A_i'$ are all zero. Therefore, the first 
$r$ rows and the first $r$ columns of $A_i'$ are all zero since $A_i'$ is psd. Now let $B_j' := U^{-1}B_jU$ for all $j=1,\ldots,q$. Then for all $i=1,\ldots,p$ and $j=1,\ldots,q$, 
$$ \langle A_i', B_j' \rangle = \textup{Tr}(U^{-1}A_iB_j U) = \langle A_i, B_j \rangle = M_{ij}. $$
However, since $A_i'$ has nonzero entries only in its bottom right $(k-r) \times (k-r)$ block,  it also follows that 
$ M_{ij} = \langle \tilde{A_i}, \tilde{B_j} \rangle$ where 
$\tilde{A_i}$ is the bottom right $(k-r) \times (k-r)$-submatrix of $A_i'$ and $\tilde{B_j}$ is the bottom right 
$(k-r) \times (k-r)$ submatrix of $B_j'$. Thus, there exists a $\PSD^{k-r}$-factorization of $M$ which is a contradiction to the fact that the psd rank of $M$ is $k$. Therefore, $\rankpsd M' \geq k+1$. 

An $\PSD^{k+1}$-factorization of $M'$ can be obtained from an $\PSD^k$-factorization 
$A_1, \ldots, A_p$ $B_1, \ldots, B_q \in \PSD^k$ of $M$ by setting 
$$ \tilde{A_i} := \left[ \begin{array}{cc} A_i & {\bf 0} \\ {\bf 0} & 0 \end{array} \right], 
\tilde{B_j} := \left[ \begin{array}{cc} B_j & {\bf 0} \\ {\bf 0} & w_j \end{array} \right],
\tilde{A} := \left[ \begin{array}{cc} {\bf 0} & {\bf 0} \\ {\bf 0} & 1 \end{array} \right], 
\tilde{B} := \left[ \begin{array}{cc} {\bf 0} & {\bf 0} \\ {\bf 0} & \alpha \end{array} \right].$$

Now consider an $\PSD^{k+1}$-factorization of $M'$ and let $B$ be the matrix associated to the last column of $M'$ in this factorization. If $\rank(B)=r$, then by the same argument as above, there exists an $\PSD^{k+1-r}$-factorization of $M$. Since $\rankpsd M = k$, $k+1-r \geq k$ or equivalently, $r \leq 1$. Since $B \neq 0$, it follows that $\rank(B)=1$.
\end{proof}

\begin{example} \label{ex:psd rank of diagonal matrix}
The psd rank of a $n \times n$ diagonal matrix with positive diagonal entries is $n$. The statement holds for $n=1$ and the general case follows by induction on $n$ and the first part of Proposition~\ref{prop:extending rank}. 
Each factor in an $\PSD^n$-factorization of such a diagonal matrix must have rank one. This 
follows by applying the second part of Proposition~\ref{prop:extending rank} to both the diagonal matrix and its transpose.
\end{example}


\section{Hadamard square roots and psd ranks of polytopes}
\label{sec:slack matrices}

In this section we derive a lower bound to the psd rank of any polytope. We begin with the following easy fact.
 
\begin{lemma} \label{lem:rank of slack matrix} 
Let $P \subset \RR^n$ be an $n$-dimensional polytope. Then a slack matrix $S_P$ has rank $n+1$.
\end{lemma}

\begin{proof} 
Let the vertices of $P$ be $p_1, \ldots, p_v$ and the facet inequalities of $P$ be $\langle a_j, x \rangle \leq \beta_j$ for $j=1, \ldots, f$. Then the corresponding $v \times f$ slack matrix 
$S_P$ has $(i,j)$-entry equal to $\beta_j - \langle a_j, p_i \rangle$, and we may factorize $S_P$ as 
$$ \left( \begin{array}{ll} 
1 & p_1 \\
\vdots & \vdots \\
1 & p_v \end{array} \right) 
\left( \begin{array}{ccc}
\beta_1 & \cdots & \beta_f\\
- a_1 & \cdots & - a_f 
\end{array} \right). $$
Since $P$ is full-dimensional and bounded, both of the factors have rank $n+1$.
\end{proof}

We now obtain a lower bound on the psd rank of a polytope.

\begin{proposition} \label{prop:lower bound on psd rank}
If $P \subset \RR^n$ is a full-dimensional polytope, then the psd rank of $P$ is at least $n+1$. Furthermore, if 
$\rankpsd P = n+1$, then {\em every} $\PSD^{n+1}$-factorization of the slack matrix of $P$ only uses rank one matrices as factors.
\end{proposition}

\begin{proof} 
The proof is by induction on $n$. If $n=1$, then $P$ is a line segment and we may assume that its vertices are 
$p_1, p_2$ and facets are $f_1, f_2$ with $p_1 = f_2$ and $p_2 = f_1$. Hence its slack matrix is a $2 \times 2$ diagonal matrix with positive diagonal entries. By the arguments in Example~\ref{ex:psd rank of diagonal matrix}, $\rankpsd S_P = 2$ 
and any $\PSD^{2}$-factorization of it uses only matrices of rank one.

Assume the first statement in the theorem holds up to dimension $n-1$ and consider a polytope $P \subset \RR^n$ of dimension $n$. Let $F$ be a facet of $P$ with vertices $p_1, \ldots, p_s$, facets $f_1, \ldots, f_t$ and slack matrix $S_F$. Suppose $f_i$ corresponds to facet $F_i$ of $P$ for $i=1, \ldots, t$. By induction hypothesis, $\rankpsd F = 
\rankpsd S_F \geq n$. Let $p$ be a vertex of $P$ not in $F$ and assume that the top left $(s+1) \times (t+1)$ submatrix of $S_P$ is indexed by $p_1, \ldots, p_s, p$ in the rows and $F_1, \ldots, F_t, F$ in the columns. Then this submatrix of $S_P$, which we will call $S_F'$, has the form 
$$ S_F' = \left( \begin{array}{cc} S_F & {\bf 0} \\ * & \alpha \end{array} \right)$$ 
with $\alpha > 0$.  By Proposition~\ref{prop:extending rank}, the psd rank of $S_F'$ is at least $n+1$ since the psd rank of $S_F$ is at least $n$. Hence, $\rankpsd P = \rankpsd S_P \geq n+1$. 

Suppose there is now a $\PSD^{n+1}$-factorization of $S_P$ and therefore of $S_F'$. By Proposition~\ref{prop:extending rank} the factor 
corresponding to the facet $F$ has rank one. Repeating the procedure for all facets $F$ and all submatrices $S_F'$ we get that all factors corresponding to the facets of $P$ in this $\PSD^{n+1}$-factorization of $S_P$ must have rank one. To prove that all factors indexed by the vertices of $P$ also have rank one, recall that the transpose of a slack matrix of $P$ is (up to row scaling) a slack matrix of the polar polytope $P^\circ$, concluding the proof.
\end{proof}

\begin{remark} \label{rmk:zero pattern}
The zero pattern in $S_P$ has been used to provide lower bounds for $\rankplus P$ (see for instance, \cite{Yannakakis, FKPT}). 
We note that the zero pattern of a slack matrix by itself is not enough to improve the lower bound on psd rank given in Proposition~\ref{prop:lower bound on psd rank}.  For example, consider the slack matrix $S_k$ of a k-gon in $\RR^2$.  Then $\rankpsd S_k$ grows to infinity as $k$ goes to infinity as shown in \cite{GPT2}.  The Hadamard square $S_k^2$, however, has the same zero pattern as $S_k$ and $\rankpsd S_k^2 \leq \rank S_k = 3$ by Lemma~\ref{lem:rank of slack matrix}.
\end{remark}

\begin{example} The {\em Birkhoff polytope} $B(n)$ is the convex hull of all $n \times n$ permutation matrices. 
It was shown in \cite{FKPT} that $\rankplus B(n) = n^2$ when $n \geq 5$. By Proposition~\ref{prop:lower bound on psd rank}, $\rankpsd B(n) \geq n^2-2n+2$. The {\em permutahedron} $\Pi(n)$ is the convex hull of the vectors 
$(\pi(1), \ldots, \pi(n))$ where $\pi$ is a permutation on $n$ letters. It was shown in \cite{Goemans} that $\rankplus \Pi(n) = O(n \textup{ log } n )$. By Proposition~\ref{prop:lower bound on psd rank}, $\rankpsd \Pi(n) \geq n$.
\end{example}

\begin{theorem} \label{thm:psdrank n+1} 
If $P \subset \RR^n$ is a full-dimensional polytope, then  $\rankpsd P = n+1$ if and only if 
$\sqrtrank S_P = n+1$.
\end{theorem}

\begin{proof}
By Proposition~\ref{prop:psd rank and Hadamard square roots}, $\rankpsd P \leq \sqrtrank S_P$. Therefore, if $\sqrtrank 
S_P  = n+1$, then by Proposition~\ref{prop:lower bound on psd rank}, the psd rank of $P$ is exactly $n+1$.

Conversely, suppose $\rankpsd P = n+1$. Then there exists a $\PSD^{n+1}$-factorization of $S_P$ which, by Proposition~\ref{prop:lower bound on psd rank}, has all factors of rank one. Thus, by Lemma~\ref{lem:rank one factors}, we have $\sqrtrank S_P \leq n+1$.  Since $\sqrtrank$ is bounded below by $\rankpsd$, we must have $\sqrtrank S_P = n+1$.
\end{proof}

Theorem~\ref{thm:psdrank n+1} says that if a full-dimensional polytope $P \subset \RR^n$ has the minimum possible psd rank $n+1$, then there must be a Hadamard square root of $S_P$ of rank $n+1$ that serves as a witness. In the next section we exhibit several classes of $n$-polytopes whose psd rank is $n+1$. 
We now give examples in the plane that show that many of the properties we have derived so far for $n$-polytopes of psd rank $n+1$ fail when psd rank is larger than $n+1$.

\begin{example} \label{ex:ngons}
Consider the pentagon $P$ in $\RR^2$ with vertices $$(0,0), (1,0), (2,1), (1,2), (0,1),$$ and a regular hexagon $H$ in $\RR^2$.  Then we have slack matrices:
\[ S_P = \left[ \begin{array}{ccccc}
0 & 4 & 12 & 4 & 0 \\
0 & 0 & 8 & 8 & 2 \\
2 & 0 & 0 & 8 & 4 \\
4 & 8 & 0 & 0 & 2 \\
2 & 8 & 8 & 0 & 0 \end{array} \right], \; 
S_H = \left[ \begin{array}{cccccc}
0 & 2 & 4 & 4 & 2 & 0 \\
0 & 0 & 2 & 4 & 4 & 2 \\
2 & 0 & 0 & 2 & 4 & 4 \\
4 & 2 & 0 & 0 & 2 & 4 \\
4 & 4 & 2 & 0 & 0 & 2 \\
2 & 4 & 4 & 2 & 0 & 0 \end{array} \right] . \]
Theorem~\ref{thm:polygons of psd rank 3} will show that these polytopes have psd rank at least four which is not the minimum possible in the plane. We make the following observations:
\begin{description}
\item[(i)] $\sqrtrank S_P > \rankpsd P$

This pentagon has psd rank four due to the $\PSD^4$-factorization given by the following matrices (the first five matrices correspond to the rows and the second five to the columns):
\[ \tiny \left[ \begin{array}{rrrr}
3 & 0 & 0 & 0 \\
0 & 1 & 1 & -1 \\
0 & 1 & 1 & -1 \\
0 & -1 & -1 & 1 \end{array} \right],
\left[ \begin{array}{rrrr}
1 & -1 & 0 & 0 \\
-1 & 1 & 0 & 0 \\
0 & 0 & 1 & -1 \\
0 & 0 & -1 & 1 \end{array} \right],
\left[  \begin{array}{rrrr}
1 & 0 & 0 & -1 \\
0 & 1 & -1 & 0 \\
0 & -1 & 1 & 0 \\
-1 & 0 & 0 & 1 \end{array} \right],
\left[ \begin{array}{rrrr}
1 & 1 & 0 & 0 \\
1 & 1 & 0 & 0 \\
0 & 0 & 1 & 1 \\
0 & 0 & 1 & 1 \end{array} \right],\]

\[ \tiny \left[ \begin{array}{rrrr}
1 & 0 & 0 & 1 \\
0 & 1 & 1 & 0 \\
0 & 1 & 1 & 0 \\
1 & 0 & 0 & 1 \end{array} \right], 
\left[ \begin{array}{rrrr}
0 & 0 & 0 & 0 \\
0 & 0 & 0 & 0 \\
0 & 0 & 1 & 1 \\
0 & 0 & 1 & 1 \end{array} \right],
\left[ \begin{array}{rrrr}
1 & 1 & 1 & 1 \\
1 & 1 & 1 & 1 \\
1 & 1 & 1 & 1 \\
1 & 1 & 1 & 1 \end{array} \right],
\left[ \begin{array}{rrrr}
1 & -1 & -1 & 1 \\
-1 & 1 & 1 & -1 \\
-1 & 1 & 1 & -1 \\
1 & -1 & -1 & 1 \end{array} \right], \]

\[ \tiny \left[ \begin{array}{rrrr}
1 & -1 & 1 & -1 \\
-1 & 1 & -1 & 1 \\
1 & -1 & 1 & -1 \\
-1 & 1 & -1 & 1 \end{array} \right],
\left[ \begin{array}{rrrr}
0 & 0 & 0 & 0 \\
0 & 1 & -1 & 0 \\
0 & -1 & 1 & 0 \\
0 & 0 & 0 & 0 \end{array} \right] . \]
One can check that $\sqrtrank S_P = 5$ in this case via the following algebraic calculation.  Create a symbolic matrix with the same zeros as a $S_P$, say 
\[ S :=  \left[ \begin{array}{ccccc}
0 & a & b & c & 0 \\
0 & 0 & d & e & f \\
g & 0 & 0 & h & i \\
j & k & 0 & 0 & l \\
m & n & o & 0 & 0 \end{array} \right]. \]
Then there is a Hadamard square root of $S_P$ of rank at most four if and only if there is a solution to the system of polynomial equations $$ \{ \textup{det}(S) = 0, \,\, a^2 = 4, \,\,b^2= 12,\,\,c^2=4, \ldots, o^2 = 8 \}.$$  Using a computer algebra package such as Macaulay2 \cite{M2}, we can see that this system of equations has no solutions. Therefore, when the psd rank of a $n$-polytope is greater than $n+1$, there need not be any Hadamard square root of the slack matrix whose rank equals the psd rank of the polytope.

\item[(ii)] $\sqrtrank S_H < \rank \sqrt[+]{S_H}$

The all-nonnegative Hadamard square root $\sqrt[+]{S_H}$ has rank $5$.  The following Hadamard square root has rank 4:
\[ \left[ \begin{array}{rrrrrr}
0 & \sqrt{2} & 2 & 2 & \sqrt{2} & 0 \\
0 & 0 & \sqrt{2} & 2 & 2 & \sqrt{2} \\
\sqrt{2} & 0 & 0 & \sqrt{2} & 2 & 2 \\
-2 & -\sqrt{2} & 0 & 0 & \sqrt{2} & 2 \\
2 & -2 & -\sqrt{2} & 0 & 0 & \sqrt{2} \\
\sqrt{2} & 2 & -2 & -\sqrt{2} & 0 & 0 \end{array} \right]. \]
Thus, it is not enough to check the positive Hadamard square root of $S_P$ to get 
$\sqrtrank S_P$.

\item[(iii)] 
Recall that if $Q$ is an $n$-dimensional polytope and $\rankpsd Q = n+1$, then $\rankpsd Q = \sqrtrank Q$ and all $\PSD^{n+1}$-factorizations of $S_Q$ have factors of rank one. However, even if $\rankpsd Q = \sqrtrank Q$, but $\rankpsd Q > n+1$, then there can be factorizations of $S_Q$ by psd matrices of size $\rankpsd Q$ in which the factors do not all have rank one as in the case of the hexagon $H$.

From above, $\sqrtrank S_H = 4$.  A $\PSD^4$-factorization of $S_H$ is gotten by assigning the following six psd matrices of rank two to the columns:
\[ \tiny \left[ \begin{array}{rrrr}
1 & -1 & 0 & 1 \\
-1 & 1 & 0 & -1 \\
0 & 0 & 1 & 0 \\
1 & -1 & 0 & 1 \end{array} \right],
\left[ \begin{array}{rrrr}
1 & 0 & 0 & 0 \\
0 & 1 & 1 & -1 \\
0 & 1 & 1 & -1 \\
0 & -1 & -1 & 1 \end{array} \right],
\left[ \begin{array}{rrrr}
1 & 1 & 1 & 0 \\
1 & 1 & 1 & 0 \\
1 & 1 & 1 & 0 \\
0 & 0 & 0 & 1 \end{array} \right], \]

\[ \tiny \left[ \begin{array}{rrrr}
1 & 1 & 0 & 1 \\
1 & 1 & 0 & 1 \\
0 & 0 & 1 & 0 \\
1 & 1 & 0 & 1 \end{array} \right],
\left[ \begin{array}{rrrr}
1 & 0 & 0 & 0 \\
0 & 1 & -1 & 1 \\
0 & -1 & 1 & -1 \\
0 & 1 & -1 & 1 \end{array} \right],
\left[ \begin{array}{rrrrr}
1 & -1 & 1 & 0 \\
-1 & 1 & -1 & 0 \\
1 & -1 & 1 & 0 \\
0 & 0 & 0 & 1 \end{array} \right],  \]

and the following six psd matrices of rank one to the rows:
\[ \tiny \left[ \begin{array}{rrrr}
1 & 1 & 0 & 0 \\
1 & 1 & 0 & 0 \\
0 & 0 & 0 & 0 \\
0 & 0 & 0 & 0 \end{array} \right],
\left[ \begin{array}{rrrr}
0 & 0 & 0 & 0 \\
0 & 1 & 0 & 1 \\
0 & 0 & 0 & 0 \\
0 & 1 & 0 & 1 \end{array} \right],
\left[ \begin{array}{rrrr}
0 & 0 & 0 & 0 \\
0 & 1 & -1 & 0 \\
0 & -1 & 1 & 0 \\
0 & 0 & 0 & 0 \end{array} \right], \]

\[ \tiny \left[ \begin{array}{rrrr}
1 & -1 & 0 & 0 \\
-1 & 1 & 0 & 0 \\
0 & 0 & 0 & 0 \\
0 & 0 & 0 & 0 \end{array} \right],
\left[ \begin{array}{rrrr}
0 & 0 & 0 & 0 \\
0 & 1 & 0 & -1 \\
0 & 0 & 0 & 0 \\
0 & -1 & 0 & 1 \end{array} \right],
\left[ \begin{array}{rrrr}
0 & 0 & 0 & 0 \\
0 & 1 & 1 & 0 \\
0 & 1 & 1 & 0 \\
0 & 0 & 0 & 0 \end{array} \right] . \]
\end{description}
\end{example}

There is no systematic algorithm to find exact psd factorizations of the type shown above. The factorizations in the above example were obtained via trial and error with a pen and paper.  We always tried to choose row factors of rank one and all factors as sparse as possible.

We now give two applications of Propositions~\ref{prop:extending rank} and \ref{prop:lower bound on psd rank}. The first yields a method to produce polytopes of psd rank $k$ from polytopes of psd rank $k-1$.

\begin{proposition} \label{prop:pyramid}
If $P \subset \RR^n$ is an $n$-dimensional pyramid over a $(n-1)$-polytope $Q$ and $\rankpsd Q = k$, then $\rankpsd 
P = k+1$.
\end{proposition}
 
\begin{proof}
Let $S_Q$ be the slack matrix of $Q$. By assumption, $\rankpsd S_Q = k$. We may assume without loss of generality that $Q$ lies in the hyperplane $x_n = 0$ and that the apex $v$ of $P$ has $v_n > 0$. The facets of $P$ that contain $v$ are in bijection with the facets of $Q$. The only other facet inequality of $P$ is $x_n \geq 0$. A slack matrix of $P$ is 
 $$ \left[ 
 \begin{array}{cc}
 S_Q & {\bf 0} \\ 
 {\bf 0} & \alpha \end{array}
  \right] $$
  where the last row is indexed by $v$ and the last column by $x_n \geq 0$. Therefore, $\alpha > 0$ and by Proposition~\ref{prop:extending rank}, the psd rank of $S_P$ is $k+1$.
\end{proof}

The following result will be used in Section~\ref{sec:examples}.

\begin{proposition} \label{prop:face psd rank}
If a polytope $P$ has a facet of psd rank $k$, then $P$ has psd rank at least $k+1$. In particular, if $\rankpsd P = n+1$ where $P \subset \RR^n$ is a $n$-polytope, then $\rankpsd F = i+1$ for every $i$-dimensional face of $P$.
\end{proposition}

\begin{proof}
The first fact is an immediate consequence of the proof of Proposition~\ref{prop:lower bound on psd rank} where we saw that if $F$ is a facet of psd rank $k$, then Proposition~\ref{prop:extending rank} can be used to construct a submatrix $S_F'$ of the slack matrix $S_P$ that has psd rank at least $k+1$. 
The second statement then follows from Proposition~\ref{prop:lower bound on psd rank}. 
\end{proof}

\section{ Families of polytopes of minimum psd rank} \label{sec:min psd rank}
\label{sec:examples}

Recall that if $P$ is an $n$-dimensional polytope in $\RR^n$ then $\rankplus P \geq n+1$. It is straightforward to see that the only $n$-dimensional polytopes of nonnegative rank $n+1$ are simplices. The psd situation is much richer with many more classes of polytopes achieving the minimum possible psd rank as we show in this section.

\begin{definition} A $n$-dimensional polytope $P \subset \RR^n$ is said to be $2$-{\em level} if it has a slack matrix all of whose entries are zero or one. Geometrically, $P$ is $2$-level if and only if for each facet of the polytope, all vertices of $P$ lie on the union of this facet and exactly one other parallel translate of the hyperplane spanning this facet.
\end{definition}

It follows from \cite{GPT1} that a $2$-level polytope in $\RR^n$ admits an $\PSD^{n+1}$-lift which can be constructed explicitly using sums of squares polynomials. 
In the language of the current paper, it follows that $n$-dimensional $2$-level polytopes have psd rank $n+1$. We can also see this directly from Theorem~\ref{thm:psdrank n+1}.

\begin{corollary} \label{cor:2level} 
Let $P$ be an $n$-dimensional $2$-level polytope in $\RR^n$. Then the psd rank of $P$ is exactly $n+1$.
Further, all the factors in any $\PSD^{n+1}$-factorization of $P$ have rank one.
\end{corollary}

\begin{proof}
Since a $2$-level polytope has a $0/1$ slack matrix $S_P$, $\rank \sqrt[+]{S_P} = \rank S_P = n+1$. Therefore, $\sqrtrank S_P = n+1$, and by Theorem~\ref{thm:psdrank n+1}, the psd rank of a $2$-level polytope equals $n+1$. The second statement follows from Proposition~\ref{prop:lower bound on psd rank}.
\end{proof}

Since any $n$-polytope with $n+1$ vertices is a simplex which is $2$-level, its psd rank is $n+1$. 
In fact, Theorem~\ref{thm:psdrank n+1} implies the following stronger result.

\begin{theorem} \label{thm:n+2 vertices}
Any full-dimensional polytope in $\RR^n$ with $n+2$ vertices has psd rank $n+1$.
\end{theorem}

\begin{proof}
Suppose $P$ is a polytope with $n+2$ vertices.  Then if $f$ is the number of facets of $P$, we have that $S_P$ is an $(n+2) \times f$ matrix of rank $n+1$.  Let $S_i$ denote the $i$th row of $S_P$.  Since $\rank S_P = n+1$, we have $\sum_{i=1}^{n+2}a_iS_i = \left( 0, \ldots,0 \right)$ for some $a_i \in \RR$.  Each column of $S_P$ must have at least $n$ zeros, so when we consider the above equation component-wise, all but at most two of the summands must be zero.  Thus, for each $j = 1, \ldots, f$, $a_{i_0}\left(S_{i_0}\right)_j +a_{i_1}\left(S_{i_1}\right)_j = 0$ for some $1 \leq i_0, i_1 \leq n+2$.  For each $a_i$ define $b_i := \text{sgn}\left(a_i\right)\sqrt{\left| a_i\right|}$.  Then $b_{i_0}\sqrt{\left(S_{i_0}\right)_j} + b_{i_1}\sqrt{\left(S_{i_1}\right)_j} = 0$.  Since this holds for each component, we have $\sum_{i=1}^{n+2}b_i\sqrt{S_i} = \left( 0, \ldots,0 \right)$.  Thus, $\sqrt[+]{S_P}$ must have rank $n+1$ and the result follows from Theorem~\ref{thm:psdrank n+1}.
\end{proof}

There are $\lfloor n^2/4 \rfloor$ distinct combinatorial types of $n$-dimensional polytopes with $n+2$ vertices \cite{Grunbaum}. In the plane, we get that all quadrilaterals have psd rank three.
In $\RR^3$, the two combinatorial types of polytopes with five vertices are the pyramid over a quadrilateral and a double simplex (bipyramid over a triangle). A quadrilateral pyramid need not be $2$-level but it is combinatorially equivalent to a pyramid over a square which is $2$-level.  By Theorem~\ref{thm:n+2 vertices}, a $n$-dimensional double simplex (bipyramid over a simplex of dimension $n-1$) has psd rank $n+1$. They are polytopes of minimum psd rank that are not combinatorially equivalent to $2$-level polytopes.

\begin{proposition} \label{prop:double simplex}
There is no $2$-level polytope that is combinatorially equivalent to a double simplex except in the plane.
\end{proposition}

\begin{proof} Let $P \subset \RR^n$ be an $n$-dimensional double simplex. Then the support of any $(n+2) \times 2n$ slack matrix of $P$ where the first and last rows correspond to the vertices acquired when taking the bipyramid over a $(n-1)$-dimensional simplex is
$$M := \left( \begin{array}{ccc|ccc}
0 & \cdots & 0 & 1 & \cdots & 1 \\
\hline
& I_n & & & I_n & \\
\hline
1 & \cdots & 1& 0 & \cdots & 0 
\end{array}
\right).$$
The rank of $M$ is $n+1$ and hence the left kernel of $M$ has dimension one and is generated by the vector $z := (1,-1,-1,\ldots,-1,-1,1) \in \RR^{n+2}$ with all entries equal to $-1$ except the first and last. Also, $P$ is combinatorially equivalent to a $2$-level polytope if and only if there is a ($2$-level) polytope with slack matrix $M$.

Suppose $M$ is the slack matrix of a $n$-dimensional polytope. Then we should be able to factorize $M$ as in the proof of 
Lemma~\ref{lem:rank of slack matrix} into the form 

$$ M = \left( \begin{array}{ll} 
1 & p_1 \\
\vdots & \vdots \\
1 & p_{n+2} \end{array} \right) 
\left( \begin{array}{ccc}
\beta_1 & \cdots & \beta_f\\
- a_1 & \cdots & - a_{2n}
\end{array} \right).$$

Call the two factors $V$ and $F$. The left kernel of $V$ is non-trivial since $V$ is a $(n+2) \times (n+1)$ matrix. 
Let $z'$ be a non-zero element in the left kernel of $V$. Then since $z'VF = 0$, it must also be that $z'M = 0$. This implies that $z'$ is a scalar multiple of $z$ and hence $z$ is in the left kernel of $V$. But looking at the first column of $V$, which is all ones, we see that $z$ can be in the left kernel of $V$ only if $n=2$.
\end{proof}

On the other hand, being combinatorially equivalent to a $2$-level polytope does not imply minimal psd rank. The regular octahedron in $\RR^3$ is a $2$-level polytope but we now show an octahedron whose psd rank is five. 

\begin{example} \label{ex:octahedra}
Consider the octahedron with vertices $$(0,0,0),(2,0,0),(0,2,0),(2,2,0),(1,1,-1),(1,2,1)$$ which has 
slack matrix:
\[ \left[ \begin{array}{cccccccc}
0 & 0 & 0 & 0 & 2 & 2 & 2 & 2 \\
0 & 2 & 0 & 2 & 0 & 0 & 2 & 2 \\
2 & 0 & 2 & 0 & 2 & 2 & 0 & 0 \\
2 & 2 & 2 & 2 & 0 & 0 & 0 & 0 \\
0 & 2 & 3 & 0 & 0 & 2 & 1 & 0 \\
3 & 0 & 0 & 2 & 2 & 0 & 0 & 1 \\ \end{array} \right] . \]
It can be checked algebraically as in Example~\ref{ex:ngons} that no Hadamard square root of this slack matrix has rank four. However, the positive Hadamard square root has rank five and hence the psd rank of this octahedron is five. 
\end{example}

\begin{remark}
We have seen that having the combinatorial type of a $2$-level polytope is not enough for minimal psd rank, while being the image under a projective transformation of a $2$-level polytope is enough. 
Proposition~\ref{prop:double simplex} shows that not all polytopes of minimal psd rank are projectively equivalent to $2$-level polytopes. Strictly weaker than being projectively equivalent to a $2$-level polytope is the existence of a positive scaling of each row
and column of $S_P$ that turns it into a $0/1$-matrix. This clearly implies minimal psd rank, and includes double simplices. So one could suppose this to be a necessary and sufficient condition for having
minimal psd rank. This turns out to be false. Consider the prism with vertices $(0,0,0)$, $(1,0,0)$, $(0,1,0)$, $(1,2,0)$, $(0,0,1)$, $(1,0,1)$, $(0,1,1)$, $(1,2,1)$ which has slack matrix
\[ \left[ \begin{array}{cccccc}
0 & 0 & 2 & 1 & 0 & 1 \\
1 & 0 & 0 & 2 & 0 & 1 \\
0 & 1 & 2 & 0 & 0 & 1 \\
1 & 2 & 0 & 0 & 0 & 1 \\
0 & 0 & 2 & 1 & 1 & 0 \\
1 & 0 & 0 & 2 & 1 & 0 \\
0 & 1 & 2 & 0 & 1 & 0 \\
1 & 2 & 0 & 0 & 1 & 0 \end{array} \right] . \]
The positive square root of this matrix has rank four, so the polytope has minimal psd rank, but it is easy to see that we can never turn the submatrix from the first two rows and the fourth and sixth columns into a $0/1$-matrix
by any scaling.
\end{remark}

In the plane we can fully characterize the polytopes of psd rank three. 

\begin{theorem} \label{thm:polygons of psd rank 3}
A convex polygon $P$ in the plane has psd rank three if and only if it has at most four vertices.
\end{theorem}

\begin{proof}
The ``if'' direction was discussed after Theorem~\ref{thm:n+2 vertices}.

Now suppose that $P$ is a convex polygon with $5$ or more vertices. By an affine transformation we can suppose $P$ has facets given by $x \geq 0$ and $y \geq 0$ with vertices on
$(0,0)$, $(1,0)$ and $(0,1)$. Let $(a,b)$ be the vertex sharing an edge with $(0,1)$ and $(c,d)$ the one sharing an edge with $(1,0)$.
These facets are then given by the two inequalities $(b-1)x - ay + a \geq 0$ and $(c-1)y-dx + d \geq 0$ respectively, so we can take the $5 \times 4$ submatrix of the slack matrix of $P$ indexed by these vertices and facets, which is then
\[S'_P=\left(\begin{array}{cccc}
0 & 0 & a & d \\
0 & 1 & 0 & d+c-1 \\
1 & 0 & a+b-1 & 0  \\
a & b & 0 & cb-b-da+d \\
c & d & bc-c-ad+a & 0 \end{array}\right). \]
It is then enough to show that every possible Hadamard square root of the $4 \times 4$ upper left portion of this matrix has rank four. This matrix is given by 
\[\left(\begin{array}{cccc}
0 & 0 & \pm\sqrt{a} & \pm\sqrt{d} \\
0 & \pm 1 & 0 & \pm\sqrt{d+c-1} \\
\pm 1 & 0 & \pm\sqrt{a+b-1} & 0  \\
\pm\sqrt{a} & \pm\sqrt{b} & 0 & \pm\sqrt{cb-b-da+d} \end{array}\right).\]
Assume this matrix has rank three.  Since the first three rows are independent, we can write the fourth row as a combination of the first three.  In such a combination, the coefficients for the first three rows must be $\pm\sqrt{a+b-1}$, $\pm\sqrt{b}$ and $\pm\sqrt{a}$, respectively. For ease of notation, let $\alpha=b(d+c-1)$ and $\beta=d(a+b-1)$.  Then $\alpha,\beta > 0$ and $\alpha \geq \beta$.  Looking at the last column, we see that
\[ \pm\sqrt{\alpha-\beta}=\pm\sqrt{\alpha}\pm\sqrt{\beta} .\]
Out of these eight possible equations, the only four that are feasible are $\pm \sqrt{\alpha-\beta}=\sqrt{\alpha}-\sqrt{\beta}$ and $\pm \sqrt{\alpha-\beta}=-\sqrt{\alpha}+\sqrt{\beta}$, all of which imply $\alpha=\beta$.  Hence, $cb-b=ad-d$ and we have that $b/(a-1)=d/(c-1)$.  Thus, the slope of the line between $(a,b)$ and $(1,0)$ equals the slope between $(c,d)$ and $(1,0)$, implying that the three are collinear and cannot all be vertices unless $(a,b)=(c,d)$.
\end{proof}

In $\RR^3$, it is more difficult to classify the convex polytopes of minimum psd rank.  We have seen that all polytopes with four or five vertices have psd rank four.  Additionally, we can say precisely which octahedra in $\RR^3$ have psd rank four.  Let $O \subset \RR^3$ be a (combinatorial) octahedron.  We say that $O$ is planar with respect to a plane $E$ if $O \cap E$ contains four vertices of $O$.  For example, the regular octahedron is planar to the $xy$, $xz$, and $yz$ planes.  A combinatorial octahedron can be planar with respect to at most three planes.  We say $O$ is \emph{biplanar} if it is planar with respect to at least two distinct planes.

\begin{theorem} \label{thm:octahedra}
An octahedron $O \subset \RR^3$ has psd rank four if and only if $O$ is biplanar.
\end{theorem}

\begin{proof}
First, assume $O$ is biplanar.  Then, by applying an affine transformation, we can assume that $O$ is planar with respect to the $xy$ plane and has vertices $(0,0,0)$, $(1,0,0)$, $(0,1,0)$, $(a,b,0)$, $(z_1,z_2,z_3)$, and $(w_1,w_2,w_3)$ where $z_3 > 0$, $w_3 < 0$, and $a+b > 1$.  

For ease of notation, let $\alpha = z_3 - w_3$, $\beta = w_1z_3 - z_1w_3$, and $\gamma = w_2z_3 - z_2w_3$.  Then $(0,0,0)$, $(a,b,0)$, $(z_1,z_2,z_3)$, $(w_1,w_2,w_3)$ are coplanar if and only if $b\beta = a\gamma$ and $(1,0,0)$, $(0,1,0)$, $(z_1,z_2,z_3)$, $(w_1,w_2,w_3)$ are coplanar if and only if $\alpha = \beta + \gamma$.  The combinatorics of $O$ dictates that these are the only possible further planarities, and 
since $O$ is biplanar, at least one of these conditions must be satisfied.

Now $O$ has slack matrix $S_O$:
\begin{small}
\[ \left[ \begin{array}{cccccccc}
0 & 0 & b & a & 0 & 0 & b & a \\
1 & 0 & 0 & a+b-1 & 1 & 0 & 0 & a+b-1 \\
0 & 1 & a+b-1 & 0 & 0 & 1 & a+b-1 & 0 \\
a & b & 0 & 0 & a & b & 0 & 0 \\
0 & 0 & 0 & 0 & \frac{-\beta}{w_3} & \frac{-\gamma}{w_3} & \frac{b(\beta-\alpha) + (1-a)\gamma}{w_3} & \frac{a(\gamma-\alpha) + (1-b)\beta}{w_3} \\ 
\frac{\beta}{z_3} & \frac{\gamma}{z_3} & \frac{b(\alpha-\beta) + (a-1)\gamma}{z_3} & \frac{a(\alpha-\gamma) + (b-1)\beta}{z_3} & 0 & 0 & 0 & 0 \end{array} \right] . \]
\end{small}

In the case $b\beta = a\gamma$ or the case $\alpha = \beta + \gamma$, row reduction shows that $\sqrt[+]{S_O}$ has rank four.  Hence, $O$ has psd rank four.

For the converse, suppose $O$ is planar to either one or zero planes.  If a planar condition is satisfied, assume it is by the vertices $v_1,v_2,v_3,v_4$.  By applying an affine transformation, we can assume that $v_1 = (0,0,1)$, $v_2 = (0,0,0)$, $v_3 = (1,0,0)$, and $v_5 = (0,1,0)$.  Let $v_4 = (z_1,z_2,z_3)$ and $v_6 = (w_1,w_2,w_3)$ where we must have 
\begin{equation} \label{eq:octahedron conditions}
z_1 < 0, \,\, w_3 > 0, \,\,1- z_1 - z_2 - z_3 > 0, \textup{ and } 1 - w_1 - w_2 - w_3 > 0
\end{equation}
to preserve the combinatorial structure. (These are not all of the required conditions, but we will use these particular ones below.)

Since $O$ cannot satisfy planarity conditions on the set of vertices $\left\{v_1,v_2,v_5,v_6\right\}$ or  $\left\{v_3,v_4,v_5,v_6\right\}$, we must have that 
\begin{equation} \label{eq:planarity violations}
w_1 \neq 0 \textup{ and } w_1 z_3 + w_2 z_3 - z_1 w_3 - z_2 w_3 + w_3 - z_3 \neq 0.
\end{equation}

We calculate the slack matrix $S_O$ and consider its $5 \times 5$ submatrix $M$ indexed by the vertices $v_1,v_2,v_3,v_5,v_6$ in the rows and the facets $F_{1,3,5}$, $F_{2,3,6}$, $F_{2,4,5}$, $F_{1,3,6}$, $F_{1,4,5}$ in the columns where $F_{i,j,k}$ is the facet defined by the vertices $v_i,v_j,v_k$.  After multiplying the rows and columns by nonnegative constants, $M$ has the form:
\begin{tiny}
\[ \left[ \begin{array}{ccccc}
0 & 1 & 1 & 0 & 0 \\
1 & 0 & 0 & 1 & 1 \\
0 & 0 & z_3 & 0 & 1-z_1-z_2-z_3 \\
0 & w_3 & 0 & 1-w_1-w_2-w_3 & 0 \\
-z_1(1-w_1-w_2-w_3) & 0 & -z_1 w_3 + w_1 z_3 & 0 & -z_1(1-w_2-w_3) + w_1(1-z_2-z_3) \end{array} \right] . \]
\end{tiny}

Now consider an arbitrary Hadamard square root $\sqrt{M}$.  For the purposes of calculating rank of $\sqrt{M}$, we can assume that the $(1,2)$, $(1,3)$, $(2,1)$, $(2,4)$, and $(2,5)$ entries of $\sqrt{M}$ are all $1$.  Let 
\[ S= \left[ \begin{array}{ccccc}
0 & 1 & 1 & 0 & 0 \\
1 & 0 & 0 & 1 & 1 \\
0 & 0 & s_1 & 0 & s_2 \\
0 & s_3 & 0 & s_4 & 0 \\
s_5 & 0 & s_6 & 0 & s_7 \end{array} \right] . \]
be a symbolic matrix corresponding to a $\sqrt{M}$ and let $\tilde{z}_1,\ldots,\tilde{w}_3$ be variables corresponding to $z_1,\ldots,w_3$.  Consider the ideal $I$ generated by the polynomials:
  \[ \left\{ \det S, \,s_1^2 - \tilde{z}_3, \ldots, s_7^2 + \tilde{z}_1(1-\tilde{w}_2-\tilde{w}_3) - \tilde{w}_1(1-\tilde{z}_2-\tilde{z}_3) \right\}. \]
Now if $\rankpsd O = 4$, then $\sqrtrank M \leq 4$ and, hence, there must exist real numbers $x_1,\ldots,x_7$ such that $(x_1,\ldots,x_7,z_1,\ldots,w_3)$ lies in $V(I)$, the variety of $I$.

The three possible planarity conditions on $O$ are given by the equations:
\[
\tilde{w}_1 = 0, \,\,\tilde{z}_2 =0, \textup{ and }   \tilde{w}_1 \tilde{z}_3 + \tilde{w}_2 \tilde{z}_3 - \tilde{z}_1 \tilde{w}_3 - \tilde{z}_2 \tilde{w}_3 + \tilde{w}_3 - \tilde{z}_3 = 0.
\]
Let $J_1, J_2, J_3$ be the ideals generated by two each of the three polynomials defining the above planarity conditions.
Then the product ideal $J := J_1*J_2*J_3$ has variety  $V(J) = V(J_1) \cup V(J_2) \cup V(J_3)$. By our planarity assumption on $O$, $(x_1,\ldots,x_7,z_1,\ldots,w_3)$ is not contained in $V(J)$.  Now $V(I) \backslash (V(J)$ is contained in the variety of the colon ideal $I:J$ \cite[Chapter 4.4, Theorem 7]{CLO} and, hence, $(x_1,\ldots,x_7,z_1,\ldots,w_3)$ vanishes on every polynomial in $I:J$. Using Macaulay2 \cite{M2}, we can compute a set of generators of $I:J$ and by elimination one sees that 
\[ f = \tilde{z}_1\tilde{w}_1\tilde{w}_3(\tilde{w}_1+\tilde{w}_2+\tilde{w}_3 - 1)(\tilde{z}_1+\tilde{z}_2+\tilde{z}_3 - 1)(\tilde{w}_1 \tilde{z}_3 + \tilde{w}_2 \tilde{z}_3 - \tilde{z}_1 \tilde{w}_3 - \tilde{z}_2 \tilde{w}_3 + \tilde{w}_3 - \tilde{z}_3) \]
lies in $I:J$. However, no choice of $z_1,\ldots,w_3$ that is required to satisfy (\ref{eq:octahedron conditions}) and (\ref{eq:planarity violations}) can vanish on $f$.  Hence, we must have $\rankpsd O \geq 5$.
\end{proof}

A {\em cuboid}, or combinatorial cube, is a polytope in $\RR^3$ that is combinatorially equivalent to a cube. Since the polars of cuboids are octahedra and psd rank is preserved under polarity, the cuboids of minimal psd rank are precisely those that are polars of biplanar octahedra.  We call these biplanar cuboids. More explicitly, these are the cuboids for which there exists two sets of four facets whose supporting hyperplanes intersect in a point (possibly at infinity).


We will now argue that there are no polytopes in $\RR^3$ of psd rank four beyond the ones we have considered above (and their polars).
Let $P$ be a polytope in $\RR^3$ of psd rank four. By Proposition~\ref{prop:face psd rank}, all the facets of $P$ must be triangles or quadrilaterals.  Further, since $\rankpsd P^\circ = 4$, each vertex of $P$ must be of {\em degree} three or four.  Recall that the degree of a vertex of $P$ is the number of edges of $P$ incident to that vertex.


\begin{lemma}\label{lem:facet restriction}
Let $P \subset \RR^3$ be a three-dimensional polytope with $\rankpsd P = 4$.  If $p$ is a vertex of $P$ of degree four, then the four facets incident to $p$ must be triangles.
\end{lemma}

\begin{proof}
Let $P$ and $p$ be as above and suppose that the four facets incident to $p$ are not all triangles.  By 
Proposition~\ref{prop:face psd rank}, one of the facets surrounding $p$ must be a quadrilateral and $P$ contains the following structure (with $p_1,\ldots,p_5$ vertices of $P$):

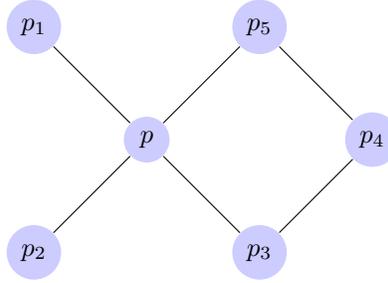
\begin{center}
\begin{tikzpicture}
  [scale=.3,auto=center,every node/.style={circle,fill=blue!20}]
  \node (p) at (6,6) {$p$};
  \node (p1) at (1,11)  {$p_1$};
  \node (p2) at (1,1)  {$p_2$};
  \node (p3) at (11,1) {$p_3$};
  \node (p4) at (16,6)  {$p_4$};
  \node (p5) at (11,11)  {$p_5$};
\foreach \from/\to in {p/p1,p/p2,p/p3,p/p5,p3/p4,p4/p5}
    \draw (\from) -- (\to);
\end{tikzpicture}
\end{center}

Let $S_P$ be a slack matrix of $P$.  Then $S_P$ is of rank four.  Further, since $P$ has minimum psd rank, there exists a Hadamard square root $\sqrt{S_P}$ of rank four. Let $M$ be the $5 \times 4$ submatrix of $\sqrt{S_P}$ indexed by $p, p_1,p_2,p_3, p_4$ in the rows and by the four facets incident to $p$ in the columns.  By scaling the columns of $\sqrt{S_P}$ by nonzero scalars, we may assume that $M$ is of the following form, with $a,b,c,d,e$ nonzero:

\[ \left[\begin{array}{cccc}
0 & 0 & 0 & 0 \\
0 & 1 & 0 & 1 \\
0 & 0 & 1 & a \\
1 & 0 & b & 0 \\
c & d & e & 0 \end{array} \right]. \]

The four rows of $\sqrt{S_P}$ and $S_P$ corresponding to the first four rows of $M$ are linearly independent by the structure of $M$.  Hence, we can write the row of $\sqrt{S_P}$ and $S_P$ corresponding to the fifth row of $M$ as a linear combination of the other four.  Thus, we can write the fifth row of $M$ and $M^2$ as a linear combination of the first four.  This results in two necessary equations: $d + ae = abc$ and $d^2 + a^2e^2 = (abc)^2$, which implies that $ade=0$, a contradiction.
\end{proof}

\begin{proposition}\label{prop:3d classification}
A polytope in $\RR^3$ of psd rank four has the combinatorial type of a simplex, quadrilateral pyramid, bisimplex, triangular prism, octahedron, or cube.
\end{proposition}

\begin{proof}
Let $P$ be a polytope in $\RR^3$ of psd rank four with $v$ vertices, $e$ edges, and $f$ facets.  Let $v_t$ and $v_q$ denote the number of vertices of degree three and four in $P$, and let $f_t$ and $f_q$ denote the number of triangular and quadrangular facets of $P$.

By double counting edges, $2e = 3f_t+4f_q$, and by considering $P^\circ$, we also see that $2e=3v_t+4v_q$.  Now using Euler's formula, $v-e+f=2$, it is easy to deduce that $v_t$ and $f_t$ are even and that $v_t+f_t=8$.  Hence, we only need to consider polytopes where $(v_t,f_t)$ equals $(0,8)$, $(2,6)$, $(4,4)$, $(6,2)$, or $(8,0)$. Further, by taking polars we need only consider the cases where $(v_t,f_t)$ equals $(0,8)$, $(2,6)$, or $(4,4)$.

When $(v_t,f_t) = (0,8)$, we have that every vertex is of degree four.  Thus, by Lemma~\ref{lem:facet restriction}, every facet must be triangular.  The only polytope in $\RR^3$ that satisfies these conditions is the octahedron.

Now suppose $(v_t,f_t) = (4,4)$.  If there are no degree four vertices, then there are only four total vertices and the polytope must be the simplex.  If there is a degree four vertex, then by Lemma~\ref{lem:facet restriction} the polytope must contain the following configuration:

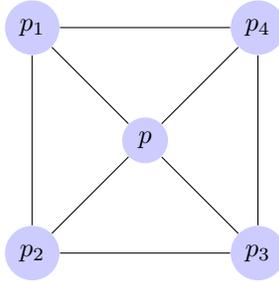
\begin{center}
\begin{tikzpicture}
  [scale=.3,auto=center,every node/.style={circle,fill=blue!20}]
  \node (p) at (6,6) {$p$};
  \node (p1) at (1,11)  {$p_1$};
  \node (p2) at (1,1)  {$p_2$};
  \node (p3) at (11,1) {$p_3$};
  \node (p4) at (11,11)  {$p_4$};
  \foreach \from/\to in {p/p1,p/p2,p/p3,p/p4,p3/p4,p4/p1,p1/p2,p2/p3}
    \draw (\from) -- (\to);
\end{tikzpicture}
\end{center}

If vertex $p_1$,$p_2$,$p_3$, or $p_4$ has degree four, then we will be forced to include too many triangular facets.  Thus, they all have degree three and the polytope is a quadrilateral pyramid.

Finally, suppose $(v_t,f_t) = (2,6)$.  Then $P$ must have a degree four vertex (call it $p$) and the configuration above is again included in the boundary complex of $P$ with the four triangles shown being facets of $P$.  Since $P$ has only two vertices of degree three, at least two of the vertices surrounding $p$ must have degree four.  Suppose two adjacent vertices among $p_1,p_2,p_3,p_4$ have degree four. Then each of them must be contained in four triangular facets which means that each such vertex is incident to two triangular facets that are not shown in the figure. But since these degree four adjacent vertices already share a facet, they can share at most one of these four extra triangular facets. This creates a total of seven triangular facets in $P$ contradicting $f_t=6$. Therefore, the two vertices of degree four among $p_1,p_2,p_3,p_4$ must be nonadjacent. As before, each is adjacent to two triangular facets that are not shown and since $f_t=6$, it must be that the two vertices share these two triangular facets. Therefore, $P$ is a bisimplex. 

Now the facts that the polar of a bisimplex is combinatorially a triangular prism, and the polar of an octahedron is a cube  completes the proof.
\end{proof}

We now immediately obtain the following theorem which gives a complete classification of polytopes in $\RR^3$ of psd rank four.

\begin{theorem} \label{thm:3d polytopes of psd rank four}
The polytopes in $\RR^3$ of psd rank four are precisely simplices, quadrilateral pyramids, bisimplicies, combinatorial triangular prisms, biplanar octahedra, and biplanar cuboids.
\end{theorem}

A major catalyst for the use of semidefinite programming in combinatorial optimization was 
the {\em Lov{\'a}sz theta body of a graph} \cite{ShannonCapacity, GLS}, denoted as $\textup{TH}(G)$,  which is a convex relaxation of the stable set polytope of a graph. Let $G = ([n],E)$ be a graph with vertex set $[n] := \{1, \ldots, n \}$ and edge set $E$. Recall that a {\em stable set} of $G$ is a subset $S \subseteq [n]$ such that for all $i,j \in S$, the pair $\{i,j\}$ is not in $E$. The {\em characteristic vector} of a stable set $S$ is $\XX^S \in \{0,1\}^n$ defined as $(\XX^S)_i = 1$ if $i \in S$ and $0$ otherwise. The {\em stable set polytope} of $G$ is the $n$-dimensional polytope 
$$ \STAB(G) := \textup{convex hull}( \XX^S \,:\, S \textup{ stable set in } G ) \subset \RR^n,$$
and $\textup{TH}(G)$ is the following projection of an affine slice of $\PSD^{n+1}$: 
$$\left\{ x \in \RR^n \,:\, \exists \left[ \begin{array}{cc} 1 & x^T \\ x & U \end{array} \right] \succeq 0 \textup{ s.t. } U_{ii} = x_i \,\,\forall \,\, i=1,\ldots,n \textup{ and } U_{ij} = 0 \,\,\forall \,\,\{i,j\} \in E \right\}.$$
Further, $\textup{TH}(G) = \STAB(G)$ if and only if $G$ is a {\em perfect graph} \cite[Chapter 9]{GLS}. Hence if $G$ is perfect, 
$\rankpsd \STAB(G) = n+1$ and the description of $\textup{TH}(G)$ gives a $\PSD^{n+1}$-lift of $\STAB(G)$. In the context of this paper, 
it is natural to ask if there are non-perfect graphs for which $\rankpsd \STAB(G) = n+1$, via other $\PSD^{n+1}$-lifts.

\begin{theorem} \label{thm:perfect graphs}
Let $G$ be a graph with $n$ vertices.  Then $\STAB(G)$ has psd rank $n+1$ if and only if $G$ is perfect.
\end{theorem}

\begin{proof}
We saw that $\rankpsd \STAB(G) = n+1$ when $G$ is a perfect graph with $n$ vertices. 
Suppose $G$ is not perfect.  By Proposition~\ref{prop:face psd rank}, it is enough to show that $\STAB(G)$ has a face that is not of minimal psd rank.  By the perfect graph theorem \cite{CRSTAnnals}, $G$ contains a {\em odd hole} or {\em odd anti-hole} $H$.  Since $\STAB(H)$ forms a face of $\STAB(G)$, we just need to show that $\STAB(H)$ is not of minimal psd rank.  

Let $H=([2m+1],E)$ and assume $H$ is an odd hole.  The anti-hole case is exactly analogous and is omitted here.  Now $\STAB(H)$ is a $(2m+1)$-dimensional polytope with facet inequalities: 
\begin{enumerate}
\item $x_i \geq 0$ for each $i \in [2m+1]$
\item $\xx_e \leq 1$ for each $e \in E$
\item $\xx_{[2m+1]} \leq m$
\end{enumerate}
where $\xx_T := \sum_{i \in T} x_i$ for every subset $T$ of $[2m+1]$ and $\xx_e := x_{i} + x_{j}$ for $e=\left\{ i,j \right\} \in E$. Let $S$ be the slack matrix of $\STAB(H)$ and let $S^\prime$ be the $(2m+3) \times (2m+3)$ submatrix of $S$ where $S^\prime$ is indexed by the stable sets \[ \left\{ \; \right\},\left\{ 1 \right\},\left\{ 2 \right\},\ldots,\left\{ 2m+1 \right\},\left\{ 1,3\right\} \] in the rows and the facets $\xx_{\left\{ 1,2\right\}} \leq 1, x_{1} \geq 0, \ldots, x_{2m+1} \geq 0, \xx_{[2m+1]} \leq m$  in the columns.  Then $S^\prime$ has the form:
\[ \left[ \begin{array}{cccccccc}
1 & 0 & 0 & 0 & 0 & \cdots & 0 & m \\
0 & 1 & 0 & 0 & 0 & \cdots & 0 & m-1 \\
0 & 0 & 1 & 0 & 0 & \cdots & 0 & m-1 \\
1 & 0 & 0 & 1 & 0 & \cdots & 0 & m-1 \\
1 & 0 & 0 & 0 & 1 & \cdots & 0 & m-1 \\
\vdots & \vdots & \vdots & \vdots & \vdots & \ddots & \vdots & \vdots \\
1 & 0 & 0 & 0 & 0 & \cdots & 1 & m-1 \\
0 & 1 & 0 & 1 & 0 & \cdots & 0 & m-2 \end{array} \right] . \]

Let $\sqrt{S^\prime}$ be an arbitrary Hadamard square root and suppose that $\rank \sqrt{S^\prime} \leq 2m+2$.  Then since the first $2m+2$ columns are linearly independent, we must have that the final column is a linear combination of the first $2m+2$.  Let $\alpha_1,\ldots,\alpha_{2m+2}$ be coefficients in such a combination.  By looking at the first, second, fourth, and last columns, we see that $\alpha_1 = \pm \sqrt{m}$, $\alpha_2 = \pm \sqrt{m-1}$, and $\alpha_4 = \pm \sqrt{m} \pm \sqrt{m-1}$.  Now by looking at the last row, we must have $\pm \alpha_2 \pm \alpha_4 = \pm \sqrt{m-2}$, which is a contradiction.  Hence, $\sqrtrank S > 2m+2$ and we have that $\STAB(H)$ is not of minimal psd rank.
\end{proof}

\bibliographystyle{plain}

\end{document}